\documentclass{amsart}

\usepackage{amssymb}
\usepackage{amsthm}
\usepackage{amsmath} 
\usepackage{amsbsy}
\usepackage{hyperref}
\usepackage{tikz}
\newtheorem{theorem}{Theorem}

\newtheorem*{proposition}{Proposition}

\newtheorem*{lem}{Lemma}

\begin{document}

\title[]{Fourier Uncertainty Principles, Scale Space Theory and the Smoothest Average} 

\thanks{S.S. is supported by the NSF (DMS-1763179) and the Alfred P. Sloan Foundation.}

\author[]{Stefan Steinerberger}
\address{Department of Mathematics, Yale University, New Haven, CT 06510, USA}
\email{stefan.steinerberger@yale.edu}

\begin{abstract} Let $f \in L^{2}(\mathbb{R}^n)$ and suppose we are interested in computing its average at a fixed scale. This is easy: we pick the density $u_{}$ of a probability distribution with mean 0 and some moment at the desired scale and compute the convolution $u_{} * f$. Is there a particularly natural choice for $u$? This question is studied in scale space theory and the Gaussian is a popular answer. We were interested whether a canonical choice for $u$ can arise from a new axiom: having fixed a scale, the average should oscillate as little as possible, i.e.
$$ u_{} = \arg\min_{u_{}} \sup_{f \in L^2(\mathbb{R}^n)} \frac{\| \nabla (u_{} *f) \|_{L^2(\mathbb{R}^n)}}{\|f\|_{L^2(\mathbb{R}^n)}}.$$
This optimal function turns out to be a minimizer of an uncertainty principle:
for $\alpha > 0$ and $\beta > n/2$, there exists $c_{\alpha, \beta,n} > 0$ such that for all $u \in L^1(\mathbb{R}^n)$
$$  \| |\xi|^{\beta} \cdot \widehat{u}\|^{\alpha}_{L^{\infty}(\mathbb{R}^n)} \cdot \| |x|^{\alpha} \cdot u \|^{\beta}_{L^1(\mathbb{R}^n)} \geq c_{\alpha, \beta,n} \|u\|_{L^1(\mathbb{R}^n)}^{\alpha + \beta}.$$
For $\beta = 1$, any nonnegative extremizer of the inequality serves as the best averaging function in the sense above, $\beta \neq 1$ corresponds to other derivatives. For $(n, \beta)=(1,1)$ we use the Shannon-Whittaker formula to prove that the characteristic function $u(x) = \chi_{[-1/2,1/2]}$ is a local minimizer among functions defined on $[-1/2,1/2]$ for $\alpha \in \left\{2,3,4,5,6\right\}$. We provide a sufficient condition for general $\alpha$ in terms of a sign pattern for the hypergeometric function $_1F_2$.
\end{abstract}
\maketitle

\section{Introduction and Motivation}

 What is the best way to partition a cake into two pieces for two different people? This question, even when properly quantified, will not have a clear universal answer. However, it is conceivable to pose a number of axioms that one wishes a cake-division rule to satisfy and study
the set of all cake-subdivision rules satisfying these axioms. This axiomatic method has been very effectively used in cooperative game theory and economics (see e.g. \cite{har, kapeller, nash, shapley}). A nice by-product of the axiomatic approach is that it moves the discussion from `what should we do?' to `what are desirable properties?' which often leads to more insight. In the same manner, we ask a question that was the original motivation of this paper.
\begin{quote}
\textbf{Question.} Let $f:\mathbb{R}^n \rightarrow \mathbb{R}$. What is the `best' way to average $f$ over a given scale? What are natural desirable properties that one could require of such an averaging procedure and which averaging procedures are characterized by these properties?
\end{quote}

In the spirit of the axiomatic method, we will pose a number of desirable properties and then investigate what these properties imply. The various symmetries of $\mathbb{R}^n$ should be reflected in the averaging method: in particular, we will focus on the special case
where
$$ \mbox{average of}~f~\mbox{in a point}~x = \int_{\mathbb{R}^n}{ f(x+y) u(y) dy},$$
where $u:\mathbb{R}^n \rightarrow \mathbb{R}_{}$ is a nonnegative, radial function with $L^1-$norm $\|u\|_{L^1(\mathbb{R}^n)}=1$. Moreover, we will assume that the averaging is supposed to happen at a fixed scale, we will do so by imposing a condition that a certain moment is fixed, i.e.
$$ \int_{\mathbb{R}^n}{ |x|^{\alpha} u(x) dx} = \mbox{fixed}.$$
However, even with all these restrictions, there are still a large number of functions $u$ that could conceivably be used.  This question has been actively studied in \textit{scale-space theory} (see \cite{babaud, linde, linde2, yu}), a theoretical branch of image processing concerned with the same question: how should one properly smooth an image? In this field, the Gaussian is the canonical choice:
\begin{quote}
``A notable coincidence between the different 
scale-space formulations that have been stated is that the Gaussian 
kernel arises as a unique choice for a large number of different 
combinations of underlying assumptions (scale-space axioms).'' (Lindeberg \cite{linde}, 1997)
\end{quote}

We were motivated by trying to understand the implications of a new \textbf{axiom}: `a convolution at a certain scale should be as smooth as possible'. Obviously, this can be interpreted in many ways -- a very natural way is to look for the function $u$ satisfying all the constants above for which the constant $c_u$ in the inequality
$$ \forall ~f \in L^2(\mathbb{R}^n)  \qquad  \| \nabla (u*f)\|_{L^2(\mathbb{R}^n)} \leq c_u \|f\|_{L^2(\mathbb{R}^n)}$$
is as small as possible. Using the Fourier-Transform, we see that, up to a universal constant $c_n$,
$$  \| \nabla (u*f)\|^2_{L^2(\mathbb{R}^n)} = c_{n} \int_{\mathbb{R}^n} |\xi|^2 |\widehat{u}(\xi)|^2 |\widehat{f}(\xi)|^2 d\xi \leq c_n\| \xi \cdot \widehat{u}(\xi) \|_{L^{\infty}(\mathbb{R}^n)}^2 \|f\|_{L^2(\mathbb{R}^n)}^2.$$
It is not too difficult to see that this constant is sharp (since $\widehat{u}(\xi)$ is continuous, we can construct a function $f$ concentrating its $L^2-$mass close to a point where $\xi \cdot \widehat{u}(\xi)$ assumes its extremum). So the question is simply: which function minimizes $ \| \xi \cdot \widehat{u}(\xi) \|_{L^{\infty}}$ among all radial functions with normalized $L^1-$mass and a normalized moment (controlling the scale)? 
This is the question that we address in this paper. However, we emphasize are many other interesting questions in the vicinity. There are  other ways of studying oscillation of a function than $\|\nabla^s f\|$ and other function spaces than $L^2$ in which one could measure the size of a function and its derivative. Finally, we note one particularly interesting problem that arises in $n=1$ dimensions when one demands $u$ to be supported in $[-\infty, 0]$. This question is of particular importance for time-series: how would one compute the average score of a function when one cannot look into the future?  This case is much less understood: there are arguments in favor of the exponential distribution \cite{schonberg, steini}, Gaussian constructions \cite{linde2} and intermediate constructions \cite{vickrey}. It would be very interesting to have a better understanding of this case, also from the perspective taken in this paper.

\section{The Results}
\subsection{An Uncertainty Principle.} We state the most general form of the statement; the case most of interest to us throughout the rest of the paper is $(n,\beta) = (1,1)$. The case $\beta \neq 1$ corresponds to either higher derivatives (if $\beta \in \mathbb{N}$) or fractional derivatives (if $\beta \notin \mathbb{N}$). We know very little about these cases. 
\begin{theorem}[Uncertainty Principle] For any $\alpha > 0$ and $\beta > n/2$, there exists $c_{\alpha, \beta,n} > 0$ such that for all $u \in L^1(\mathbb{R}^n)$
$$  \| |\xi|^{\beta} \cdot \widehat{u}\|^{\alpha}_{L^{\infty}(\mathbb{R}^n)} \cdot \| |x|^{\alpha} \cdot u \|^{\beta}_{L^1(\mathbb{R}^n)} \geq c_{\alpha, \beta,n} \|u\|_{L^1(\mathbb{R}^n)}^{\alpha + \beta}.$$
\end{theorem}
This inequality shows that fixing the $L^1-$mass to be $\|u\|_{L^1(\mathbb{R}^n)} = 1$ and fixing any moment leads to a universal lower bound on
how small $\||\xi|^{\beta} \cdot \widehat{u}(\xi)\|_{L^{\infty}(\mathbb{R}^n)}$ can be. This shows that our axiom for the averaging operation is meaningful: for any
averaging function $u$ (having fixed scale and $L^1-$norm) there is indeed a frequency $\xi$ such that $u * \exp(i \xi x)$ is not all that small.
Somewhat to our surprise, we were not able to locate this uncertainty principle among the large number of results that have been obtained in this area (see e.g. \cite{amrein, babenko, beckner, beckner2, bene, bene2, bened, bourgain, cohn, cow, dreier, dreier2, ehm, feig, folland, gneiting, gobber, gobber2, gon, gon2, gon3, gorb, gorb2, hardy, heinig, hirsch, hogan, johan, martini, mor}). Indeed, it seems that most uncertainty principles have the lower bound in $L^2$. Two somewhat related inequalities are given by a special case of the Cowling-Price uncertainty principle \cite{cow} stating that for any $\alpha > 0$ and $\beta > 1/2$
$$  \| |\xi|^{\beta} \cdot \widehat{u}\|^{\alpha + \frac12}_{L^{\infty}(\mathbb{R})} \cdot \| |x|^{\alpha} \cdot u \|^{\beta - \frac12}_{L^1(\mathbb{R})} \geq c_{\alpha, \beta} \|u\|_{L^2(\mathbb{R})}^{\alpha + \beta}.$$
and an inequality of Laeng \& Morpurgo \cite{laeng} 
$$  \| \xi \cdot \widehat{u}\|^{2}_{L^{2}(\mathbb{R})} \cdot \| |x|^{2} \cdot u \|_{L^1(\mathbb{R})} \geq c_{} \|u\|_{L^1(\mathbb{R})} \|u\|_{L^2(\mathbb{R})}^2$$
which has some resemblance to our inequality for $(n,\alpha, \beta)= (1,2,1)$
$$  \| \xi \cdot \widehat{u}\|^{2}_{L^{\infty}(\mathbb{R})} \cdot \| |x|^{2} \cdot u \|_{L^1(\mathbb{R})} \geq c_{} \|u\|_{L^1(\mathbb{R})}^3.~~~~~~~~\quad~~~\qquad$$

\subsection{The Characteristic Function.} From now on, we will restrict ourselves to trying to understand the extremizer in the case $(n,\beta) = 1$. Other cases may be just as interesting.
Considering the initial motivation of finding the `best' kernel for the purpose of smoothing functions, an interesting choice is given by the characteristic function of an interval that is symmetric around the origin -- using the dilation symmetry, we can restrict ourselves to the case
$$ u(x) = \chi_{[-1/2, 1/2]}.$$
This function does indeed lead to a very small constant in the uncertainty principle: in particular, as soon as $\alpha \geq 1.38$, the characteristic function leads to a smaller constant than the Gaussian. We prove that it is a local minimizer among even functions $u:[-1/2,1/2] \rightarrow \mathbb{R}$ for some parameters.
\begin{theorem}[Characteristic Function as Local Minimizer] Let $(n,\beta) =(1,1)$ and $\alpha \in \left\{2,3,4,5,6\right\}$. The characteristic function $u(x) = \chi_{[-1/2,1/2]}(x)$ is a local minimizer in the class of even, smooth functions $f:[-1/2, 1/2] \rightarrow \mathbb{R}$.
\end{theorem}

The proof is based on the lucky confluence of several factors:
\begin{enumerate}
\item if $ u(x) = \chi_{[-1/2, 1/2]},$ then $\xi \cdot \widehat{u}(\xi)$ assumes its extrema on $\mathbb{Z} + 1/2$. 
\item $\widehat{u}(\xi)$ is band-limited: its Fourier transform is supported on $[-1/2, 1/2]$ 
\item the Shannon-Whittaker reconstruction formula allows us to reconstruct such a band-limited function from equally spaced  function values as long as we can sample with density at least 1
\item and all the arising computations can be carried out.
\end{enumerate}

The proof of Theorem 2 requires a Lemma that may be interesting in its own right.
Let $f: [-1/2, 1/2] \rightarrow \mathbb{R}$ be an even, smooth function. We introduce the quantity
$$ \max(\widehat{f})  =  \max \left\{ \sup_{k \in \mathbb{N}}  \left(2k+\frac{1}{2}\right)\widehat{f}\left(2k+\frac{1}{2}\right), -\inf_{k \in \mathbb{N}}  \left(2k + \frac{3}{2}\right)\widehat{f}\left(2k +\frac{3}{2}\right)   \right\}.
$$ 
This quantity arises naturally in the stability analysis. As it turns out, we have the following sharp inequality (equality is attained for constant functions).

\begin{lem} Let $\alpha \in \left\{2,3,4,5,6\right\}$ and let $f:[-1/2, 1/2] \rightarrow \mathbb{R}$ be smooth and even. We have
$$ \max(\widehat{f}) \geq \frac{\alpha+1}{\alpha \pi}\int_{-1/2}^{1/2} \left(1-|2x|^{\alpha}\right) f(x) dx.$$
\end{lem}
This seems to be quite a curious statement -- it would be interesting to understand it better; we can verify it in some special cases, $\alpha \in \left\{2,3,4,5,6\right\}$, but it does seem like it should be a special instance of a more general principle.
It is quite conceivable that the Lemma holds for all integers $\alpha \geq 2$ or possibly even for all real numbers $\alpha \geq 2$. A necessary condition is given in the next section.

\subsection{A Sign Pattern in $_1F_2$?} At this point, it is natural to wonder about the restriction $\alpha \in \left\{2,3,4,5,6\right\}$. As far as we can tell,
any case $\alpha \in \mathbb{N}$ can be decided by a finite procedure that consists of analyzing the sign pattern of an explicit polynomial: this poses no difficulty for $\alpha \in \left\{2,3,4,5,6\right\}$ and it does seem like it could be easily done for individual larger values of $\alpha$ as well. However, we have not found a common
mechanism by which all of them can be established simultaneously (or, put differently, a reason \textit{why} they should have such a sign pattern). This seems to hinge on an interesting
sign pattern structure in a hypergeometric function.

\begin{proposition} Let $\alpha > 0$. We define, for integers $k \geq 1$, the sequence
$$ a_k = ~_1F_2\left(  \frac{1+\alpha}{2}; \frac{3}{2}, \frac{3 + \alpha}{2}; -\frac{\pi^2}{16} (2k-1)^2 \right).$$
\emph{If} $a_k \geq 0$ for odd values of $k$ and $a_k \leq 0$ for even values of $k$, then for all smooth, even functions $f:[-1/2, 1/2] \rightarrow \mathbb{R}$,
$$ \max(\widehat{f}) \geq \frac{\alpha+1}{\alpha \pi}\int_{-1/2}^{1/2} \left(1-|2x|^{\alpha}\right) f(x) dx.$$
Moreover, the characteristic function $\chi_{[-1/2,1/2]}(x)$ is a local minimizer among smooth functions supported in $[-1/2,1/2]$ for that value of $\alpha$ and $\beta=1$.
\end{proposition}
For any $\alpha \in \mathbb{N}$, the hypergeometric function reduces to a trigonometric polynomial that is not terribly difficult to analyze. However, we have not found a uniform way of treating all parameters of $\alpha$. It is also conceivable that the result holds for all $\alpha \geq 2$. Numerically, it seems to fail for $\alpha < 2$ (though this becomes harder to check as $\alpha$ approaches 2). The hypergeometric function is given by
$$ ~_1F_2\left( 1 + \frac{\alpha}{2}; \frac{3}{2}, \frac{3 + \alpha}{2}; x\right) = \sum_{n=0}^{\infty} \frac{1 + \frac{\alpha}{2}}{1 + \frac{\alpha}{2} + n} \frac{1}{(\frac{3}{2})_n} \frac{x^n}{n!}.$$
It is hard to see sign patterns from this form. We also have the identity (e.g. \cite{cho})
$$ \int_{0}^{x} J_{\frac{1}{2}}(x) x^{\alpha - \frac{1}{2}} dx = c_{\alpha} \cdot x^{\alpha+1} ~ _1F_2\left(  \frac{\alpha + 1}{2}; \frac{3}{2}, \frac{3 + \alpha}{2}; -\frac{x^2}{4} \right)$$
where $J_{1/2}$ is the Bessel function of order $1/2$ and $c_{\alpha} > 0$ is a constant. This relates the problem to the oscillation behavior of a Bessel function. Askey \cite{askey} remarks that for $\alpha=1$, there is no sign change. Other identities exist: introducing a modified Bessel function
$$ \mathcal{J}_{\alpha}(x) =~_0F_1\left(\alpha+1; -\frac{x^2}{4} \right) = \Gamma(\alpha+1) \left( \frac{x}{2} \right)^{-\alpha} J_{\alpha}(x),$$
we have the following identity (from a more general result in Cho \& Yun \cite{cho})
\begin{align*}
~_1F_2\left(  \frac{1+ \alpha}{2}; \frac{3}{2}, \frac{3 + \alpha}{2}; -\frac{x^2}{4}\right) &= \mathcal{J}_{1/2}\left(\frac{x}{2}\right)^2 \\
&+\sum_{n=1}^{\infty} \frac{2n+1}{n+1} \frac{1}{(3/2)_n^2} \frac{((1-\alpha)/2)_n}{((3-\alpha)/2)_n} \left( \frac{x}{4} \right)^{2n} \mathcal{J}_{n + \frac12}^2 \left(\frac{x}{2} \right).
\end{align*}

 We also refer to Askey \cite{askey2}, Cho \& Yun \cite{cho}, Fields \& Ismail \cite{fields} and Gasper \cite{gasper}.
It seems that there are known criteria that can be used to prove that such expressions do not change sign. In contrast, we are interested in highly controlled sign changes.

\subsection{Open Problems.} There are many open problems, we only list a few.
\begin{enumerate}
\item Does the uncertainty principle admit an extremizer? Is it compactly supported? Is it possible to show that for some parameters  $n,\alpha, \beta$ that the maximizer $u$ has Fourier decay $|\widehat{u}(\xi)| \sim |\xi|^{-\beta}$? For small values of $\beta$, this would imply that a maximizer need not be continuously differentiable.
\item Is the extremizer given by the characteristic function when $\beta =1$ and $\alpha \geq 2$? Or maybe for integer $\alpha \geq 2$? Is it a global extremizer among functions $u:[-1/2, 1/2] \rightarrow \mathbb{R}$ that do not vanish in $[-1/2,1/2]$? 
\item Is it true that for any $\alpha \geq 2$ (or maybe $2 \leq \alpha \in \mathbb{N}$?), the sequence
$$ a_k = ~_1F_2\left(  \frac{1+\alpha}{2}; \frac{3}{2}, \frac{3 + \alpha}{2}; -\frac{\pi^2}{16} (2k-1)^2 \right)$$
alternates sign?
\item What can be said about the case $\beta \neq 1$? For $\| |\xi|^{\beta} \cdot \widehat{u}\|_{L^{\infty}}$ to be finite, we require $|\widehat{u}(\xi)| \lesssim (1+|\xi|^{\beta})^{-1}$ which guarantees improved regularity for larger $\beta$. How does the regularity of the extremizer depend on $\beta$?
\item What can be said about the extremizer when $n=1$ and we restrict $u$ to be supported on the half line $[-\infty, 0]$? This is relevant when one is unable to look into the future; for an example from economics, see \cite{vickrey}.
\item These questions are just as interesting in higher dimensions but it is less clear what one could expect an extremizer to look like. It is not clear whether the characteristic function of a disk plays a similar role -- its Fourier transform is connected to the Bessel function which already arose here as well in connection with $_1 F_2$: are there other sign identities attached to it or are these connections restricted to the one-dimensional case? Is the alternating sign pattern observed for $_1 F_2$ a special instance of a more general phenomenon in higher dimensions?
\end{enumerate}

\section{Proofs}

\subsection{Proof of Theorem 1} Uncertainty principles are often a consequence of some hidden form of compactness; our proof is in a similar spirit. We first show that the inequality is invariant under multiplication with scalars and dilation. This allows us to assume without loss of generality that
$$\|u\|_{L^1(\mathbb{R}^n)} = 1 \qquad \mbox{and} \qquad  \| |x|^{\alpha} \cdot u\|_{L^1(\mathbb{R}^n)}=1$$ 
and it remains to show that $\| |\xi|^{\beta} \widehat{u}\|_{L^{\infty}(\mathbb{R}^n)}$ is not too small. The inequality is only interesting when the quantity is finite. Then we can use $\| \widehat{u}\|_{L^{\infty}(\mathbb{R}^n)} \leq \|u\|_{L^1(\mathbb{R}^n)} = 1$ close to the origin and $|\widehat{u}(\xi)| \lesssim |\xi|^{-\beta}$ away from the origin to conclude that $\widehat{u} \in L^2(\mathbb{R}^n)$ and thus $u \in L^2(\mathbb{R}^n)$.
The normalization 
$$\| |x|^{\alpha} \cdot u\|_{L^1(\mathbb{R}^n)}= 1 = \|  u\|_{L^1(\mathbb{R}^n)}$$
 implies that a nontrivial amount of $L^1-$mass is distance at most $\sim_{\alpha} 1$ from the origin. This fact combined with the Cauchy-Schwarz inequality shows that $\|u\|_{L^2} \gtrsim_{\alpha} 1$. The condition $|\widehat{u}(\xi)| \lesssim |\xi|^{-\beta}$ implies that the $L^2-$mass cannot be located at arbitrarily high frequencies (depending on $\alpha, \beta$) since $|\xi|^{-2 \beta}$ is integrable when $\beta > n/2$. If some of the $L^2-$mass is in a bounded region around the origin, then $\| |\xi|^{\beta} \widehat{u}\|_{L^{\infty}}$ is not too small unless it is all concentrated around the origin which is not possible because $\| \widehat{u}\|_{L^{\infty}(\mathbb{R}^n)} \leq \|u\|_{L^1(\mathbb{R}^n)} = 1$ concluding the argument. 

\begin{proof}[Proof of Theorem 1] We first note the behavior of the inequality under rescaling by constants and dilations. If 
$$v(x) = c \cdot u(x/L) \quad \mbox{for some} \quad c,L > 0,$$
 then
\begin{align*}
\| |\xi|^{\beta} \cdot \widehat{v}(\xi)\|_{L^{\infty}}^{\alpha} &= \| |\xi|^{\beta} \cdot \left( c L^n \widehat{u}(L \xi)\right) \|_{L^{\infty}}^{\alpha} =
c^{\alpha} \| |\xi|^{\beta} L^{\beta}  L^{n-\beta} \widehat{u}(L \xi) \|_{L^{\infty}}^{\alpha}\\
&= c^{\alpha} L^{(n-\beta)\alpha} \| |L \xi|^{\beta} \widehat{u}(L \xi) \|_{L^{\infty}}^{\alpha} = c^{\alpha} L^{(n-\beta)\alpha} \| \xi \cdot \widehat{u}(\xi) \|_{L^{\infty}}^{\alpha}
\end{align*}
as well as
\begin{align*}
 \| |x|^{\alpha} \cdot v \|^{\beta}_{L^1(\mathbb{R}^n)} &= c^{\beta} \left( \int_{\mathbb{R}^n} \left||x|^{\alpha} u\left(\frac{x}{L}\right)\right| dx\right)^{\beta} \\
&= c^{\beta} \cdot L^{\alpha \beta} \left(  \int_{\mathbb{R}^n} \left| \left|\frac{x}{L}\right|^{\alpha} u\left(\frac{x}{L}\right)\right| dx\right)^{\beta} \\
&= c^{\beta} \cdot L^{\alpha \beta + n\beta} \cdot \| |x|^{\alpha} \cdot u\|^{\beta}_{L^1(\mathbb{R}^n)}
\end{align*}
and
\begin{align*}
 \|v\|_{L^1(\mathbb{R}^n)}^{\alpha + \beta} &= c^{\alpha + \beta} L^{n(\alpha + \beta)} \|u\|^{\alpha + \beta}_{L^1(\mathbb{R}^n)}.
 \end{align*}
Thus, the inequality is invariant under multiplication with scalars and dilation. 
We use these symmetries to assume without loss of generality that
$$\|u\|_{L^1(\mathbb{R}^n)} = 1 \qquad \mbox{and} \qquad  \| |x|^{\alpha} \cdot u\|_{L^1(\mathbb{R}^n)}=1.$$ 
These two identities combined imply with Markov's inequality that for any $y > 0$,
$$1 = \int_{\mathbb{R}^n}{|x|^{\alpha} |u(x)| dx} \geq  y^{\alpha} \int_{|x| \geq y}{ |u(x)| dx}$$
implying that there is some mass around the origin
$$ \int_{|x| \leq y}{|u(x)| dx} \geq 1- \frac{1}{y^{\alpha}}$$
and, in particular, for $Y=10^{1/\alpha}$, we have
$$ \int_{|x| \leq Y}{|u(x)| dx} \geq \frac{9}{10}.$$
We note that
$$ |\widehat{u}(\xi)| \leq \min\left\{ 1, \frac{\||\xi|^{\beta} \cdot \widehat{u}\|_{L^{\infty}(\mathbb{R}^n)}}{|\xi|^{\beta}} \right\},$$
where the first inequality follows from $\|\widehat{u}\|_{L^{\infty}(\mathbb{R}^n)} \leq \|u\|_{L^1(\mathbb{R}^n)}$ and the second one is merely the definition of the $L^{\infty}-$norm. As soon as $\beta > n/2$, this shows that
\begin{align*}
\int_{\mathbb{R}^n}{ | \widehat{u}(\xi)|^2 d\xi} &\lesssim 1 + \||\xi|^{\beta} \cdot \widehat{u}\|^2_{L^{\infty}(\mathbb{R}^n)} \int_{1}^{\infty} \frac{1}{|\xi|^{2 \beta}} |\xi|^{n-1} d\xi \\
&\lesssim 1 + \||\xi|^{\beta} \cdot \widehat{u}\|^2_{L^{\infty}(\mathbb{R}^n)}.
\end{align*}
In particular, if $\||\xi|^{\beta} \cdot \widehat{u}\|^2_{L^{\infty}(\mathbb{R}^n)}$ is finite (the only case of interest here), then $\widehat{u} \in L^2(\mathbb{R}^n)$ and thus $u \in L^2(\mathbb{R}^n)$.
Using H\"older's inequality, we get that, using $\omega_n$ to denote the volume of the unit ball in $\mathbb{R}^n$,
$$ \frac{9}{10} \leq \int_{|x| \leq Y}{|u(x)| dx} \leq  \omega_n^{1/2} |Y|^{n/2} \left(  \int_{|x| \leq Y}{u(x)^2 dx} \right)^{1/2}$$
and thus
 $$ \int_{\mathbb{R}^n}{ |\widehat{u}(\xi)|^2 d\xi } = \int_{\mathbb{R}^n}{u(x)^2 dx} \geq  \int_{|x| \leq Y}{u(x)^2 dx} \geq \frac{1}{\omega_n Y^{n}}.$$
Our goal is to show that $\| |\xi|^{\beta} \cdot \widehat{u}\|_{L^{\infty}(\mathbb{R}^n)}$ cannot be arbitrarily small. If $\| |\xi|^{\beta} \cdot \widehat{u}\|_{L^{\infty}(\mathbb{R}^n)} \geq 1$, then we have achieved the goal. We can therefore assume without loss of generality that $\| |\xi|^{\beta} \cdot \widehat{u}\|_{L^{\infty}(\mathbb{R}^n)} \leq 1$. Then
$$ |\widehat{u}(\xi)| \leq \min\left\{ 1, \frac{1}{|\xi|^{\beta}} \right\}$$
implies, for any $c_1 > 0$,
$$ \int_{|\xi| \geq c_1}{|\widehat{u}(\xi)|^2} \leq n\omega_n \int_{c_1}^{\infty}{\frac{|\xi|^{n-1}}{|\xi|^{2\beta}} d\xi} \leq  \frac{2n\omega_n}{2\beta -1}\frac{1}{c_1^{2\beta -1}}.$$
This can be made arbitrarily small by making $c_1$ sufficiently large. Using
\begin{align*}
 \int_{|\xi| \leq c_1}{ |\widehat{u}(\xi)|^2 d\xi } &=  \int_{\mathbb{R}^n}{ |\widehat{u}(\xi)|^2 d\xi } - \int_{|\xi| \geq c_1}{ |\widehat{u}(\xi)|^2 d\xi }\\
 &\geq \frac{1}{\omega_n Y^n} - \int_{|\xi| \geq c_1}{ |\widehat{u}(\xi)|^2 d\xi }
\end{align*}
we see that for some constant $c_1=c_1(Y,\beta,n)$ depending only on $Y$ (and thus only on $\alpha$), $\beta$ and $n$,
$$ \int_{|\xi| \leq c_1}{ |\widehat{u}(\xi)|^2 d\xi }   \geq \frac{1}{2\omega_n Y^n}.$$
Using $|\widehat{u}(\xi)| \leq \|u\|_{L^1(\mathbb{R}^n)}=1$, we deduce that 
$$\int_{|\xi| \leq c_1}{ |\widehat{u}(\xi)| d\xi } \geq \int_{|\xi| \leq c_1}{ |\widehat{u}(\xi)|^2 d\xi } \geq  \frac{1}{2\omega_n Y^n}$$
and that for suitable $0 < c_2 < c_1$ depending only on $Y$ and $n$,
$$ \int_{c_2 < |\xi| \leq c_1}{ |\widehat{u}(\xi)| d\xi } \geq \frac{1}{4\omega_n Y^n} > 0.$$
This, in turn, shows that
\begin{align*}
\frac{1}{4\omega_n Y^n} &\leq \int_{c_2 \leq |\xi| \leq c_1}{ |\widehat{u}(\xi)| d\xi } \leq \frac{1}{c_2^{\beta}} \int_{c_2 \leq |\xi| \leq c_1}{ |\xi|^{\beta} |\widehat{u}(\xi)| d\xi } \leq  \frac{1}{c_2^{\beta}} \| |\xi|^{\beta} \cdot \widehat{u}(\xi) \|_{L^{\infty}}.
 \end{align*}
Since all the arising constants depend only on $\alpha, \beta$ and $n$, the result follows. \end{proof}

\section{Proof of Theorem 2}
We first perform a local stability analysis to understand the type of statement we need to prove. We will then state a slight reformulation of the Shannon-Whittaker reconstruction formula and prove the desired stability result in the simplest possible case $\alpha = 2$. Indeed, in this case all the computations can be carried out in closed form. We will then give a proof of the general case which will mimic the proof of the $\alpha = 2$ case while bypassing the evaluation of one of the integrals.

\subsection{Local Stability Analysis.}
We first perform a local stability analysis of the uncertainty principle for $(n,\alpha,\beta) = (1,\alpha,1)$ around the function
$$ u(x) = \chi_{[-1/2, 1/2]}.$$
Since we are interested in sharp constant, we need to specify which normalization of the Fourier transform we use: it will be $$ \widehat{u}(\xi) = \int_{\mathbb{R}}u(x) e^{-2 \pi i \xi x} dx \qquad \mbox{leading to} \qquad \widehat{u}(\xi) = \frac{\sin{(\pi \xi)}}{\pi \xi}.$$
Let $f$ be an even function compactly supported in $[-1/2,1/2]$. We analyze the behavior of the inequality under replacing $u$ by $u + \varepsilon f$ as $\varepsilon \rightarrow 0$. We observe that $\xi \cdot \widehat{u}(\xi)$ assumes the extremal values $\pm \pi^{-1}$ and, more precisely,
$$\xi \cdot \widehat{u}(\xi) = \begin{cases} \pi^{-1} \qquad &\mbox{if}~\xi = 2n + \frac{1}{2} \\ -\pi^{-1} \qquad &\mbox{if}~\xi = 2n+\frac{3}{2}.\end{cases}$$
This allow us to determine that for any even, smooth function $f:[-1/2,1/2] \rightarrow \mathbb{R}$, as $\varepsilon \rightarrow 0$ and up to lower-order terms,
\begin{align*}
 \| \xi \cdot (\widehat{u}(\xi)  + \varepsilon \widehat{f}(\xi)) \|_{L^{\infty}} &= \frac{1}{\pi} + \varepsilon \max(\widehat{f}) + \mbox{l.o.t.},
 \end{align*}
 where $\max(\widehat{f})$ is an abbreviation for 
 $$ \max(\widehat{f})=  \max \left\{ \sup_{k \in \mathbb{N}}  \left(2k+\frac{1}{2}\right)\widehat{f}\left(2k +\frac{1}{2}\right),- \inf_{k \in \mathbb{N}} \left(2k + \frac{3}{2}\right)\widehat{f}\left(2k +\frac{3}{2}\right)   \right\}.
$$ 
The other two terms are easy to analyze since $f$ is smooth and thus
$$ \| |x|^{\alpha} ( u+ \varepsilon f) \|_{L^1} =\| |x|^{\alpha} \|_{L^1([-1/2,1/2])} + \varepsilon \int_{-1/2}^{1/2} |x|^{\alpha} f(x) dx + \mbox{l.o.t.} $$
and
$$ \| u+ \varepsilon f \|_{L^1}^{\alpha + 1} = 1 + (\alpha+1) \varepsilon \int_{-1/2}^{1/2} f(x) dx + \mbox{l.o.t.}$$
Moreover, the constant $c$ in the equation
$$  \| \xi \cdot \widehat{u}(\xi)  \|_{L^{\infty}}^{\alpha} \cdot \| |x|^{\alpha} \cdot u\|_{L^{\infty}} = c \|u\|_{L^1}^{\alpha + 1}$$
is easily computed to be
$$ c=  \pi^{-\alpha} \cdot \int_{-1/2}^{1/2}{ |x|^{\alpha} dx} = \frac{1}{(2\pi)^{\alpha}} \frac{1}{\alpha + 1}.$$
This shows that local stability at order $\varepsilon$ is equivalent to 
$$ \frac{1}{\alpha+1} \frac{1}{2^{\alpha}}\frac{\alpha \varepsilon}{\pi^{\alpha-1}}  \max(\widehat{f})  + \frac{1}{\pi^{\alpha}}  \varepsilon \int_{-1/2}^{1/2} |x|^{\alpha} f(x) dx \geq \frac{\varepsilon }{(2\pi)^{\alpha}} \int_{-1/2}^{1/2}{f(x) dx}.$$
This can be rewritten as 
$$ \max(\widehat{f}) \geq \frac{\alpha + 1}{\pi \alpha} \int_{-1/2}^{1/2}{(1-|2x|^{\alpha}) f(x) dx}.$$

\subsection{The Shannon-Whittaker Reconstruction/Interpolation Formula.} 
The Shannon-Whittaker reconstruction formula \cite{shannon, whittaker}, first formulated by Kotelnikov \cite{kotelnikov} (see L\"uke \cite{luke}), states that if $f$ is compactly supported in $[-1/2, 1/2]$, then
its Fourier transform is completely determined from its values at the integers and
$$ \widehat{f}(\xi) = \sum_{k \in \mathbb{Z}}{ \widehat{f}(k) \frac{\sin{(\pi(\xi-k))}}{\pi(\xi-k)}}.$$
Heuristically put, it states that a compactly supported function is determined by its Fourier coefficients
(adapted to the interval of corresponding length) which correspond to the values of the Fourier transform at equally
spaced points. We will use a shifted version
$$ \widehat{f}(\xi) = \sum_{k \in \mathbb{Z}}{ \widehat{f}\left( k - \frac12\right) \frac{\sin{\left(\pi(\xi-k+\frac12\right))}}{\pi\left(\xi-k+\frac12\right)}}.$$
\begin{proof}[Proof of the shifted version.] The representation follows quite easily from the symmetries of the Fourier transform. Let us consider the function
$$g(x) = e^{i \pi x} f(x).$$
Naturally, if $f$ is supported on $[-1/2,1/2]$, then so is $g$.
Then we have
$$ \widehat{g}(\xi) =  \sum_{k \in \mathbb{Z}}{ \widehat{g}(k) \frac{\sin{(\pi(\xi-k))}}{\pi(\xi-k)}}.$$
However, we also have
$$ \widehat{g}(\xi) = \widehat{f}\left(\xi-\frac{1}{2}\right).$$
Therefore
\begin{align*}
 \widehat{f}\left(\xi-\frac{1}{2}\right) =  \widehat{g}(\xi) =  \sum_{k \in \mathbb{Z}}{ \widehat{g}(k) \frac{\sin{(\pi(\xi-k))}}{\pi(\xi-k)}} =   \sum_{k \in \mathbb{Z}}{ \widehat{f}\left(k - \frac12\right) \frac{\sin{(\pi(\xi-k))}}{\pi(\xi-k)}}.
 \end{align*}
\end{proof}

\subsection{The case $\alpha = 2$.} The purpose of this section is to explain the argument in its simplest possible setting. The case is particularly interesting because all of the arising quantities can be computed in closed form allowing for a very explicit argument. 
We will show that for smooth, even $f:[-1/2,1/2] \rightarrow \mathbb{R}$
$$ \max(\widehat{f}) \geq \frac{3}{2\pi}\int_{-1/2}^{1/2} (1-4x^2) f(x) dx.$$
\begin{proof}
We use the Plancherel identity
$$ \int_{\mathbb{R}} f(x) g(x) dx = \int_{\mathbb{R}} \widehat{f}(\xi) \widehat{g}(\xi)d\xi$$
to write
 $$\int_{-1/2}^{1/2} (1-4x^2) f(x) dx =  \int_{\mathbb{R}}^{}  \frac{2 \sin{(\pi \xi)} - 2\pi \xi \cos{(\pi \xi)}}{\pi^3 \xi^3}\widehat{f}(\xi) d\xi.$$
We use the Shannon-Whittaker reconstruction formula to decompose
$$ \widehat{f}(\xi) = \sum_{k \in \mathbb{Z}} \widehat{f}\left(k-\frac{1}{2}\right) \frac{\sin{(\pi\left(\xi -k +\frac12\right))}}{\pi \left( \xi - k + \frac12\right)}$$
allowing us to write
 $$\int_{-1/2}^{1/2} (1-4x^2) f(x) dx  = \sum_{k \in \mathbb{Z}}{ a_k \widehat{f}\left(k-\frac12\right)},$$
 where
$$ a_k = \int_{\mathbb{R}}^{}  \frac{2 \sin{(\pi \xi)} - 2\pi \xi \cos{(\pi \xi)}}{\pi^3 \xi^3} \frac{\sin{\pi\left(\xi-k+\frac12\right)}}{\pi \left( \xi-k +\frac12\right)}d\xi.$$
We now evaluate $a_k$. Abbreviating $h(x) = \max\left\{1-4x^2,0\right\}$, we can also write
$$ a_k =  \int_{\mathbb{R}}^{}  \widehat{h}(\xi) \frac{\sin{\pi\left(\xi-k+\frac12\right)}}{\pi \left( \xi-k +\frac12\right)}d\xi.$$
We use the Shannon-Whittaker formula once more to express $\widehat{h}$ and obtain
\begin{align*}
a_k &= \int_{\mathbb{R}}^{}  \widehat{h}(\xi) \frac{\sin{\pi\left(\xi-k+\frac12\right)}}{\pi \left( \xi-k +\frac12\right)}d\xi\\
&= \int_{\mathbb{R}}^{}  \left( \sum_{m \in \mathbb{Z}}{ \widehat{h}\left( m - \frac12\right) \frac{\sin{\left(\pi(\xi-m+\frac12\right))}}{\pi\left(\xi-m+\frac12\right)}} \right) \frac{\sin{\pi\left(\xi-k+\frac12\right)}}{\pi \left( \xi-k +\frac12\right)}d\xi\\
&= \widehat{h}\left(k - \frac12\right) = \int_{-1/2}^{1/2} (1-4x^2) \cos{\left(2 \pi  \left( k - \frac12 \right) x\right)} dx \\
&= \frac{16}{\pi^3} \frac{(-1)^{k+1}}{(2k-1)^3}.
\end{align*}
This shows
$$\int_{-1/2}^{1/2} (1-4x^2) f(x) dx =  \frac{16}{\pi^3}\sum_{k\in \mathbb{Z}}  \frac{(-1)^{k+1}}{(2k-1)^3} \widehat{f}\left(k-\frac12\right).$$
We also note that, since $f$ is even and real-valued, $\widehat{f}(x) = \widehat{f}(-x)$, and thus
$$ \widehat{f}\left(k+1-\frac12\right) = \widehat{f}\left(-k-\frac12\right)$$
which allows us to group positive and negative integers into
\begin{align*}
\frac{16}{\pi^3}\sum_{k\in \mathbb{Z}}  \frac{(-1)^{k+1}}{(2k-1)^3} \widehat{f}\left(k-\frac12\right) &=
\frac{16}{\pi^3}\sum_{k=1}^{\infty} \widehat{f}\left(k-\frac12\right) \left(  \frac{(-1)^{k+1}}{(2k-1)^3}  + \frac{(-1)^{(-k+1)+1}}{(2(-k+1)-1)^3} \right)
\\ 
&=\frac{32}{\pi^3}\sum_{k=1}^{\infty} \widehat{f}\left(k-\frac12\right)   \frac{(-1)^{k+1} }{(2k-1)^3}  
\end{align*}
Let us now fix the variable $\max(\widehat{f})$ via
 $$ \max(\widehat{f})=  \max \left\{ \sup_{k\in \mathbb{N}}  \left(2k+\frac{1}{2}\right)\widehat{f}\left(2k +\frac{1}{2}\right), -\inf_{k \in \mathbb{N}}  \left(2k + \frac{3}{2}\right)\widehat{f}\left(2k +\frac{3}{2}\right)   \right\}.
$$ 
This means that for all $k\in \mathbb{N}$
$$ \widehat{f}\left(2k +\frac{1}{2}\right) \leq \frac{ \max(\widehat{f})}{2k+\frac{1}{2}} \quad \mbox{and} \quad  \widehat{f}\left(2k +\frac{3}{2}\right) \geq - \frac{ \max(\widehat{f})}{2k + \frac32}.$$
We can now maximize the sum by estimating
\begin{align*}
\frac{32}{\pi^3}\sum_{k=1}^{\infty} \widehat{f}\left(k-\frac12\right)   \frac{(-1)^{k+1} }{(2k-1)^3}   
&= \frac{32}{\pi^3}\sum_{k=1 \atop k~{\tiny \mbox{odd}}}^{\infty} \widehat{f}\left(k-\frac12\right)   \frac{1 }{(2k-1)^3}  \\
&-\frac{32}{\pi^3}\sum_{k=1 \atop k~{\tiny \mbox{even}}}^{\infty} \widehat{f}\left(k-\frac12\right)   \frac{1 }{(2k-1)^3}  \\
&\leq \frac{32}{\pi^3}\sum_{k=1 \atop k~{\tiny \mbox{odd}}}^{\infty} \frac{ \max(\widehat{f})}{k-\frac12} \frac{1 }{(2k-1)^3} \\
&+\frac{32}{\pi^3}\sum_{k=1 \atop k~{\tiny \mbox{even}}}^{\infty} \frac{ \max(\widehat{f})}{k-\frac12}   \frac{1 }{(2k-1)^3} \\
&= \frac{32}{\pi^3}\sum_{k=1 \atop k}^{\infty} \frac{ \max(\widehat{f})}{k-\frac12} \frac{1 }{(2k-1)^3} 
\end{align*}
This sum can be combined into one sum resulting in
$$ \frac{32}{\pi^3}\sum_{k=1}^{\infty} \widehat{f}\left(k-\frac12\right)   \frac{(-1)^{k+1} }{(2k-1)^3}    \leq 
\frac{64 \max(\widehat{f})}{\pi^3}\sum_{k=1}^{\infty}  \frac{1}{(2k-1)^4}.$$
We have the generalized zeta function identity
$$ \sum_{k=1}^{\infty}  \frac{1}{(2k-1)^4} = \frac{\pi^4}{96}$$
and thus
$$\int_{-1/2}^{1/2} (1-4x^2) f(x) dx \leq \frac{3 \pi}{2}  \max(\widehat{f}) \qquad \mbox{as desired.}$$
\end{proof}

\subsection{The general case.} The general case requires an additional ingredient: an oscillating sign pattern in the hypergeometric function $_1 F_2$.
\begin{lem} Let $k \in \mathbb{N}$ and consider the integral
$$ \int_{-1/2}^{1/2} (1-|2x|^{\alpha}) e^{-2 \pi i \left( k - \frac12 \right) x} dx.$$
This integral has the same sign as
$$ a_k=~_1F_2\left(  \frac{1+\alpha}{2}; \frac{3}{2}, \frac{3 + \alpha}{2}; -\frac{\pi^2}{16} (2k-1)^2 \right).$$
If $\alpha \in \left\{2,3,4,5,6\right\}$, then $a_k$ is positive for odd $k$ and negative for even $k$.
\end{lem}
\begin{proof} Since $1-|2x|^{\alpha}$ is even, it is easy to see that the imaginary part vanishes. It remains to understand the sign of the integral
$$  \int_{-1/2}^{1/2} (1-|2x|^{\alpha}) \cos{\left(2 \pi  \left( k - \frac12 \right) x\right)} dx.$$
Integration by parts leads to the integral
$$ I = \frac{2^{\alpha + 1} \alpha}{2\pi (k-1/2)} \int_0^{1/2}  x^{\alpha-1} \sin{\left(2 \pi  \left( k - \frac12 \right) x\right)}.$$
We conclude the first step of the argument by noting that the term in front of the integral is positive and that the integral
evaluates to 
$$I = \frac{\alpha}{\alpha+1}~_1F_2\left(  \frac{1+\alpha}{2}; \frac{3}{2}, \frac{3 + \alpha}{2}; -\frac{\pi^2}{16} (2k-1)^2 \right).$$
We now consider the special cases, If $\alpha = 2$, then
$$ \int_0^{1/2}  x^{\alpha-1} \sin{\left(2 \pi  \left( k - \frac12 \right) x\right)} = \frac{ (-1)^{k+1}}{ \pi^2  (2k - 1)^2}.$$ 
If $\alpha = 3$, then
$$ \int_0^{1/2}  x^{\alpha-1} \sin{\left(2 \pi  \left( k - \frac12 \right) x\right)} = \frac{\pi (-1)^{k+1} (2k-1) -2}{(2k-1)^3 \pi^3} .$$
If $\alpha = 4$, then
$$ \int_0^{1/2}  x^{\alpha-1} \sin{\left(2 \pi  \left( k - \frac12 \right) x\right)}  = \frac{3}{4} \frac{(-1)^{k+1} (\pi^2 (2k-1)^2 - 8)}{(2k -1)^4 \pi^4}.$$
If $\alpha = 5$, then
$$ \int_0^{1/2}  x^{\alpha-1} \sin{\left(2 \pi  \left( k - \frac12 \right) x\right)}  = \frac{(-1)^k (2k-1) \pi ((2k-1)^2 \pi^2 - 24) -48}{2 (2k-1)^5 \pi^5}.$$
If $\alpha = 6$, then
$$ \int_0^{1/2}  x^{\alpha-1} \sin{\left(2 \pi  \left( k - \frac12 \right) x\right)}  = \frac{5}{16}\frac{(-1)^{k+1}   ( 384 - 48(1-2k)^2 \pi^2 + (1-2k)^4 \pi^4)}{ \pi^6 (2k-1)^6}.$$
For all these explicit expressions, the claim is easily verified.
\end{proof}
\begin{proof}[Proof of Theorem 2] We now assume that $\alpha > 0$ is a real number for which
$$ a_k = \int_{-1/2}^{1/2} (1-|2x|^{\alpha}) \cos{\left(2 \pi  \left( k - \frac12 \right) x\right)} dx.$$
satisfies $a_{2k+1} \geq 0$ and $a_{2k+2} \leq 0$ for all $k \geq 0$.
Under these assumptions, we will show that
$$ \max(\widehat{f}) \geq \frac{\alpha+1}{\alpha \pi}\int_{-1/2}^{1/2} (1-|2x|^{\alpha}) f(x) dx$$
which, by the reasoning in \S 4.1, is equivalent to local stability of the minimizer.
Interpreting the variable 
$$ a_k = \int_{-1/2}^{1/2} (1-|2x|^{\alpha}) e^{-2 \pi i \left( k - \frac12 \right) x} dx $$
as the Fourier transform of $h(x) = \max\left\{0, 1 - |2x|^{\alpha}\right\}$ evaluated at $\mathbb{Z} + 1/2$, we use the Shannon-Whittaker reconstruction formula to write the integral
$$ I =  \int_{-1/2}^{1/2} (1-|2x|^{\alpha}) f(x) dx $$
as
\begin{align*}
I &= \int_{\mathbb{R}} \widehat{h}(\xi) \widehat{f}(\xi) d\xi \\
 &= \int_{\mathbb{R}}  \left( \sum_{k \in \mathbb{Z}}{ a_k\frac{\sin{\left(\pi(\xi-k+\frac12\right))}}{\pi\left(\xi-k+\frac12\right)}} \right)  \left( \sum_{k \in \mathbb{Z}}{ \widehat{f}\left( k - \frac12\right) \frac{\sin{\left(\pi(\xi-k+\frac12\right))}}{\pi\left(\xi-k+\frac12\right)}} \right) d\xi.
\end{align*}
Orthogonality leads to cancellation of off-diagonal terms and we obtain
$$   \int_{-1/2}^{1/2} (1-|2x|^{\alpha}) f(x) dx  = \sum_{k \in \mathbb{Z}} a_k \widehat{f}\left(k - \frac{1}{2}\right).$$
As above, we note that $a_k = a_{-k+1}$ and, since $f:[-1/2,1/2] \rightarrow \mathbb{R}$ is even,
$$ \widehat{f}\left(k+1-\frac12\right) = \widehat{f}\left(-k-\frac12\right)$$
allowing us to write
$$ \int_{-1/2}^{1/2} (1-|2x|^{\alpha}) f(x) dx = 2\sum_{k =1}^{\infty} a_k \widehat{f}\left( k - \frac12 \right).$$
 Introducing the variable $\max(\widehat{f})$ via
 $$ \max(\widehat{f})=  \max \left\{ \sup_{k \in \mathbb{N}}  \left(2k+\frac{1}{2}\right)\widehat{f}\left(2k +\frac{1}{2}\right), -\inf_{k \in \mathbb{N}}  \left(2k + \frac{3}{2}\right)\widehat{f}\left(2k +\frac{3}{2}\right)   \right\},
$$ 
we have for all $k\in \mathbb{N}$
$$ \widehat{f}\left(2k +\frac{1}{2}\right) \leq \frac{ \max(\widehat{f})}{2k+\frac{1}{2}} \quad \mbox{and} \quad  \widehat{f}\left(2k +\frac{3}{2}\right) \geq - \frac{ \max(\widehat{f})}{2k + \frac32}.$$
Moreover, by assumption, we have $a_k > 0$ for odd $k$ and $a_k < 0$ for even $k$. This allows us to write
\begin{align*}
 \int_{-1/2}^{1/2} (1-|2x|^{\alpha}) f(x) dx &\leq \max(\widehat{f}) 4\sum_{k=1}^{\infty}{ \frac{|a_k|}{2k-1}} \\
 &= \max(\widehat{f})  4\sum_{k=1}^{\infty}{ \frac{a_k (-1)^{k+1}}{2k-1}}.
 \end{align*}
It remains to understand this infinite sum. Recalling that $a_k$ are defined as Fourier coefficients of the function $h(x) = \max\left\{0, 1-|2x|^{\alpha} \right\}$, we can write 
\begin{align*}
 4 \sum_{k=1}^{\infty}{ \frac{a_k (-1)^{k+1}}{2k-1}} &=  4 \sum_{k=1}^{\infty}{ \frac{ (-1)^{k-1}}{2k+1}} \int_{-1/2}^{1/2} (1-|2x|^{\alpha}) e^{-2 \pi i \left( k - \frac12 \right) x} dx\\
 &=4  \int_{-1/2}^{1/2}   (1-|2x|^{\alpha})  \sum_{k=1}^{\infty}{ \frac{ (-1)^{k+1}}{2k-1}}e^{-2 \pi i \left( k - \frac12 \right) x} dx.
\end{align*}
This infinite sum can be evaluated. Note that
\begin{align*}
\sum_{k=1}^{\infty}{ \frac{ (-1)^{k+1}}{2k-1}}e^{-2 \pi i \left( k - \frac12 \right) x}  &= e^{i \pi x}\sum_{k=1}^{\infty}{ \frac{ (-1)^{k+1}}{2k-1}}e^{-2 \pi i k x}  \\
&= \arctan{(e^{-i \pi x})}\\
&= \frac{i}{2}  \log{\left(\frac{i+e^{-i \pi x}}{i-e^{-i \pi x}}\right)}
\end{align*}
We have, for $-1/2 < x < 1/2$ that
$$\arctan{(e^{-i \pi x})} = \frac{\pi}{4} + \mbox{odd and purely imaginary function}.$$
which simplifies the sum to
\begin{align*}
  4\sum_{k=1}^{\infty}{ \frac{a_k (-1)^{k+1}}{2k+1}}  &= \pi \int_{-1/2}^{1/2} (1-|2x|^{\alpha})dx  \\
  &=  \pi  \left(1 - \frac{1}{\alpha + 1}\right) = \frac{\alpha \pi}{\alpha + 1}.
  \end{align*}
Altogether, we have seen that
 \begin{align*}
  \int_{-1/2}^{1/2} (1-|2x|^{\alpha}) f(x) dx &\leq 4\max(\widehat{f})  \sum_{k=1}^{\infty}{ \frac{a_k (-1)^{k+1}}{2k-1}} =  \max(\widehat{f}) \frac{\alpha \pi}{\alpha + 1}
 \end{align*}
 which is the desired result.
\end{proof}

\end{document}